\documentclass[12pt]{amsart}
\usepackage{graphicx}
\usepackage{tikz}
\usepackage[dvipsnames]{xcolor}
\usepackage{amsmath}
\usepackage{caption}
\usepackage[top=1.5in, bottom=1.5in, left=1.5in, right=1.5in]{geometry}
\usepackage{float} 
\usepackage{mathrsfs}
\theoremstyle{plain}
\newtheorem{theorem}{Theorem}[section]

\newtheorem{lemma}[theorem]{Lemma}
\newtheorem{proposition}[theorem]{Proposition}%
\theoremstyle{definition}%
\newtheorem{remark}{Remark}%
\newtheorem{definition}{Definition}

\usepackage{fancyhdr}
\usepackage{amsmath,amssymb,mathrsfs}
\begin{document}

\title[Double Signs of Hamiltonian Circles  in Doubly Signed Complete Graphs]{Double Signs of Hamiltonian Circles  in Doubly Signed Complete Graphs}\footnote{}

\author{Xiyong Yan}
\address{89 Park Ave,  Binghamton, NY,  USA,  13903.}
\email{xiyongyan@gmail.com}
\begin{abstract}
    We study Hamiltonian circles in the doubly signed complete graph 
\(\Sigma_n = (K_n, \sigma, \mathbb{F}_2^2)\). A circle’s double sign is defined as the sum of its edge labels. I establish conditions under which Hamiltonian circles realize all four possible double signs and prove that this occurs  when the set of triangle double signs contains at least three distinct values. The proof is based on an analysis of triangle bases of the binary cycle space, structural properties of 
$K_4$ subgraphs, and explicit Hamiltonian constructions. 
\end{abstract}

\maketitle
\section{Introduction}

Signed graphs, introduced by Harary in the 1950s, assign to each edge a sign from the group $\mathbb{F}_2$. A central object of study is the sign of a circle, defined as the sum of its edge signs. Our research question originates in the problem collection “Negative (and Positive) Circles in Signed Graphs” \cite{1}, where the author asks: if a signed graph is unbalanced and has a Hamiltonian circle, must it contain a negative Hamiltonian circle? A positive one? In \cite{2} I prove that a signed complete graph on $n$ vertices contains both a positive and a negative Hamiltonian circle if and only if it also contains both a positive and a negative triangle.

In this paper we consider the case of \emph{doubly signed graphs}, where each edge of the complete graph $K_n$ is labeled by an element of  $\mathbb{F}_2^2$. For a circle $C \subseteq K_n$, its \emph{double sign} is the sum of the labels of its edges. Every circle therefore has one of four possible values, and understanding which values occur is a natural problem. Triangles are of particular importance since they form a basis for the binary cycle space $Z_2(K_n)$ and thus determine the behavior of all circles.

Our focus is on Hamiltonian circles in $\Sigma_n = (K_n,\sigma,\mathbb{F}_2^2)$, where $n>5$. Determining the range of their double signs turns out to depend only on the diversity of double signs realized by triangles. The main theorem shows that  if triangles are restricted to at most two signs, then the Hamiltonian circles cannot achieve the full set (See Lemma \ref{lemma22} and the remark 1). Conversely, if the triangles of $\Sigma_n$ produce at least three distinct double signs, then Hamiltonian circles realize all four possible values in $\mathbb{F}_2^2$ (See Lemma \ref{lemmab},\ref{lemmac}).

\section{Triangle Double Signs: One, Two, or Three Cases}

The (binary) cycle space  of an undirected graph is the set of its even-degree subgraphs, denoted as \(Z_2(G)\), where the undirected graph is \(G=(V,E)\).

Let \(\Sigma_n = (K_n, \sigma, \mathbb{F}_2^2)\), with vertices \(v_1, v_2, \ldots, v_n\), where \(\sigma \colon E \to \mathbb{F}_2^2\). Choose a triangle basis of the  cycle space \(Z_2(K_n)\):
\[
\mathcal{B} = \{ {T_{1ij}} : 2 \le i < j \le n \} , \quad \text{where the triangle } {T_{1ij}} := \triangle\{v_1, v_i, v_j\}.
\]

\begin{lemma}\label{lemma22}
If the double signs of all triangles in \(\Sigma_n\) are either \(x\) or \(y\), where \(x, y \in \mathbb{F}_2^2\), then \(\Sigma_n\) does not contain all possible double signs of Hamiltonian circles.
\end{lemma}

\begin{proof}
Let \(H\) be any Hamiltonian circle of \(\Sigma_n\). Note that the vertex \(v_1\) is in \(H\), and \(H\) contains the sequence of vertices \(v_1, v_2', v_3', \ldots, v_n'\), where the distinct vertices \( v_2', v_3', \ldots, v_n' \in \{v_2, v_3, \ldots, v_n\}\). 

We can decompose \(H\) into triangles \(T_{1i(i+1)} := \triangle \{v_1, v_i', v_{i+1}'\}\), for \(i = 2, 3, \ldots, n - 1\).

Then,
\[
\sigma({H}) = \sum_{i=2}^{n-1} \sigma(T_{1i(i+1)}) = a x + b y,
\]
since all triangles have double signs either \(x\) or \(y\), where \(a, b \in \{0, 1, \ldots, n - 2\}\) and \(a + b = n - 2\).

We now consider two cases based on the parity of \(n\):

Case 1. \(n\) is odd. Then \(n-2\) is odd, so the sum  \(a x + b y\) must be either \(x\) or \(y\). 
Thus, \(\Sigma_n\) cannot contain all possible double signs of Hamiltonian circles.

Case 2. \(n\) is even. Then \(n-2\) is even, so the possible values of \(a x + b y\) are limited to \(x+y\) or \(e = (0,0)\). Again, \(\Sigma_n\) cannot contain all possible double signs of Hamiltonian circles.
\end{proof}

\begin{remark}\label{rem:one}
If all triangles have the same double sign $x$, then
\[\sigma(H)=\sum_{i=2}^{n-1}\sigma(T_{1i(i+1)})=t\,x\] with \(t\in\{0,1\}\).
\end{remark}

\begin{definition}[Normalization at a vertex]
To \emph{normalize} a vertex \(v\) means to perform a switching so that every edge
incident with \(v\) is labeled by the identity \(e=(0,0)\).
\end{definition}

\begin{proposition}
    If \(\Sigma_n\) contains only three distinct  double signs among all triangles, then after normalizing a vertex $v$ of \(\Sigma_n\),   every edge of \(\Sigma_n\backslash v\) has a double sign among these three types.
\end{proposition}

\begin{proof}
    Let $v_iv_j$ be any edge in \(\Sigma_n\backslash v\). Then, $$ \sigma(\triangle v_iv_jv)=\sigma(v_iv_j)+e+e=\sigma(v_iv_j).$$
\end{proof}

\begin{lemma}\label{lemma1}
If \(\Sigma_n\) has at most three distinct triangle double signs, then every \(K_4\subseteq \Sigma_n\) has at most two distinct triangle double signs.
\end{lemma}

\begin{proof}[Proof (Contrapositive)]
Suppose there exists a $K_4 \subseteq \Sigma_n$ whose four triangles realize at least three distinct
double signs. Choose three triangles from this $K_4$ with distinct double signs
$x,y,z \in \mathbb{F}_2^2$. The fourth triangle then has double sign $x+y+z$.
Since $x,y,z$ are distinct, the element $x+y+z$ is the
element of $\mathbb{F}_2^2$ not in $\{x,y,z\}$. Hence $\Sigma_n$ exhibits four distinct
triangle double signs, contradicting the assumption that there are at most three.
\end{proof}

\begin{lemma}\label{lemma5}
If \(\Sigma_n\) contains only three distinct  double signs among all triangles, then the basis \(\mathcal{B}\) contains triangles realizing each of these three signs. 
\end{lemma}
\begin{proof}
Let the three distinct double signs be \(x,y,z\in\mathbb F_2^2\). Suppose no triangle in \(\mathcal{B}\) has double sign \(x\).
Then, Without loss of generality, say some triangle \(T_{234}\) outside \(\mathcal{B}\) has double sign \(x\).
Consider the induced \(K_4\) on \(\{v_1,v_2,v_3,v_4\}\). The  triangle \(T_{123}\) has double sign
\(\alpha\in\{y,z\}\). Observe that $\sigma(T_{123})+\sigma(T_{234})=\sigma(T_{124})+\sigma(T_{134})$. That is $\alpha+x=\sigma(T_{124})+\sigma(T_{134})$. By Lemma~\ref{lemma1}, \(K_4\) cannot realize three distinct double signs,
so one of \(T_{124}\) or \(T_{134}\) must have double sign \(x\), contradicting the assumption
that \(\mathcal{B}\) contains no triangle of double sign \(x\). The same argument applies to \(y\) and \(z\).
\end{proof}

\begin{definition}
We say that three triangles \(T_{1ab},T_{1bc},T_{1cd} \in \mathcal{B}\) are \emph{consecutive} 
if they share successive edges at the common vertex \(v_1\), that is,
\(T_{1ab}\) and \(T_{1bc}\) share the edge \(v_1 v_b\), and
\(T_{1bc}\) and \(T_{1cd}\) share the edge \(v_1 v_c\),
with \(a,b,c,d\) distinct.
\end{definition}

\begin{lemma}\label{lemmab}

If \(\Sigma_n\) contains only three distinct  double signs among all triangles, then \(\Sigma_n\) contains all possible double signs of Hamiltonian circles.
\end{lemma}

\begin{proof}

By Lemma~\ref{lemma5}, \(\mathcal{B}\) contains all three distinct  double signs among its triangles. 

We consider two cases. 

Case 1. Suppose there exist three consecutive triangles in $\mathcal{B}$ with distinct double signs \(x,y,z\in\mathbb{F}_2^2\); Without loss of generality, label their five vertices \(v_1,v_2,v_3,v_4,v_5\) as in the figure \ref{three1} .

\begin{figure}[H]
    \centering
\begin{tikzpicture}[scale=2]
  \coordinate (v1) at (-1,0.85);
  \coordinate (v2) at (-1,-0.2);
  \coordinate (v3) at (0.6,-.2);
  \coordinate (v4) at (0.6,0.5);
 \coordinate (v5) at (0.5,1.2);
  \draw[thick] (v2) -- (v1) -- (v3) -- cycle; 
  \draw[thick] (v3) -- (v1) -- (v4) -- cycle; 
  \draw[thick] (v2) -- (v1) -- (v4);     
   \draw[thick] (v1) -- (v5) -- (v4);  
   

  \node at (-0.5,0.2) {\(x\)};
  \node at (0.3,0.3) {\(y\)};
  \node at (0.4,0.8) {\(z\)};
  
  \node[above] at (v1) {\(v_1\)};
    \node[below] at (v2) {\(v_2\)};
\node[below] at (v3) {\(v_3\)};
\node[right] at (v4) {\(v_4\)};
\node[right] at (v5) {\(v_5\)};

\end{tikzpicture}
 \caption{}
\label{three1}

\end{figure}

Normalize at \(v_5\). Then the edge \(v_1v_4\) has double sign \(z\). 
There are exactly two isomorphism types of assignments for the induced subgraph \(K_4\) on \(\{v_1,v_2,v_3,v_4\}\), where distinct $r,t\in\{x,y\}$, as follows :

\begin{figure}[H]
    \centering
\begin{tikzpicture}[scale=3, every node/.style={}]
  \coordinate (A1) at (0,1); 
  \coordinate (B1) at (0,0); 
  \coordinate (C1) at (1,0); 
  \coordinate (D1) at (1,1); 

  \coordinate (A2) at (2,1); 
  \coordinate (B2) at (2,0); 
  \coordinate (C2) at (3,0); 
  \coordinate (D2) at (3,1); 

  \draw[thick] (A1) -- (B1) -- (C1) -- (D1) -- cycle;
  \draw[thick] (A1) -- (C1);
 \draw[thick] (B1) -- (D1);
  \draw[thick] (A2) -- (B2) -- (C2) -- (D2) -- cycle;
  \draw[thick] (A2) -- (C2);
\draw[thick] (B2) -- (D2);
  \node[above left] at (A1) {\(v_1\)};
  \node[below left] at (B1) {\(v_2\)};
  \node[below right] at (C1) {\(v_3\)};
  \node[above right] at (D1) {\(v_4\)};
  \node at (-0.1,0.5) {\(r\)};
  \node at (0.3,0.8) {\(t\)};
  \node at (0.5,1.1) {\(z\)};
  \node at (1.1,0.5) {\(z\)};
  \node at (0.5,-.1) {\(t\)};
  \node at (0.6,0.7) {\(z\)};

  \node[above left] at (A2) {\(v_1\)};
  \node[below left] at (B2) {\(v_2\)};
  \node[below right] at (C2) {\(v_3\)};
  \node[above right] at (D2) {\(v_4\)};
  \node at (3.1,0.75) {\(y\)};
 
  \node at (2.5,1.1) {\(z\)};
  \node at (2.3,0.8) {\(z\)};
  \node at (2.5,-.1) {\(z\)};
  \node at (1.9,0.5) {\(x\)};
 \node at (2.7,0.8) {\(z\)};
\end{tikzpicture}

 \caption{}
 \label{three2}
\end{figure}

To show that there are exactly two isomorphism classes of assignments, we distinguish two cases.
Observe that $\sigma(v_1v_4)=z$ and $\sigma(T_{134})=y$, so   $\sigma(v_1v_3)+\sigma(v_3v_4)+z=y$.  Since each edge sign is either $x,y$ or $z$ in $\Sigma_n\backslash v_5$, we have $\sigma(v_3v_4)=y$ or $z$. 

Case A.  $\sigma(v_3v_4)=y$. Then, $\sigma(v_1v_3)=z$. Now $\sigma(T_{123})=x$, that is $\sigma(v_1v_3)+\sigma(v_1v_2)+\sigma(v_2v_3)=x$, this implies $\sigma(v_1v_2)+\sigma(v_2v_3)=x+z.$
Thus, $\sigma(v_1v_2)=x$ or $z$. If $\sigma(v_1v_2)=x$, then $\sigma(v_3v_2)=z$.   By Lemma \ref{lemma1}, $\sigma(T_{124})=x$ or $y$,  this forces $\sigma(v_2v_4)=z$. This is isomorphic to    the right panel of Figure \ref{three2}.      If $\sigma(v_1v_2)=z$, then $\sigma(v_3v_2)=x$.   By Lemma \ref{lemma1} again, this forces $\sigma(v_2v_4)=x$ or $y$. This is isomorphic to the left panel of Figure \ref{three2}.

Case B. $\sigma(v_3v_4)=z$. Then, $\sigma(v_1v_3)=y$. With a similar reason, we have $\sigma(v_1v_2)+\sigma(v_2v_3)=x+y.$.  
Thus, $\sigma(v_1v_2)=x$ or $y$. If $\sigma(v_1v_2)=x$, then $\sigma(v_3v_2)=y$.   By Lemma \ref{lemma1}, this forces $\sigma(v_2v_4)=z$. This is isomorphic to the left panel of Figure \ref{three2}.      If $\sigma(v_1v_2)=y$, then $\sigma(v_3v_2)=x$.  This forces $\sigma(v_2v_4)=z$. This is isomorphic to the left panel of Figure \ref{three2}.

\vspace{0.5cm}

Now, we will construct four Hamiltonian circles with distinct double signs for the left figure above.  
For the right figure, we use a similar method to construct four Hamiltonian circles with distinct double signs and will omit the details.  

Let \(p_{6n}\) denote the path \(v_6 v_7 \ldots v_n\).  Connect \(v_2\) to \(v_6\) and \(v_4\) to \(v_n\).  
This yields two Hamiltonian circles in \(\Sigma_n \setminus v_5\):  
\[
H_1 := v_4 v_1 v_3 v_2 v_6 v_7 \ldots v_n, 
\quad 
H_2 := v_4 v_3 v_1 v_2 v_6 v_7 \ldots v_n.
\]
Let $H := v_4 v_1 v_2 p_{6n}$ and set $\phi(H) = k \in \mathbb{F}_2^2$. 
Then
\[
\phi(H_1) = k + x, \qquad \phi(H_2) = k + y.
\]

For an edge $e_{ij} = v_i v_j$ of $H_r$ ($r \in \{1,2\}$), let $H_r^{ij} := H_r \Delta T_{5ij}$.
We have:
\begin{align*}
\phi(H_{1}^{14}) &= \phi(H_1) + \phi(T_{5\,1\,4}) = k + x + z,\\
\phi(H_{2}^{34}) &= \phi(H_2) + \phi(T_{5\,3\,4}) = k + y + z,\\
\phi(H_{2}^{31}) &= \phi(H_2) + \phi(T_{5\,3\,1}) = k + y + y = k,\\
\phi(H_{2}^{21}) &= \phi(H_2) + \phi(T_{5\,2\,1}) = k + y + x.
\end{align*}
Because $x,y,z$ are distinct, the four elements 
\[
k+x+z,\quad k+y+z,\quad k,\quad k+y+x
\]
are all distinct. 
Thus $H_{1}^{14}, H_{2}^{34}, H_{2}^{31}, H_{2}^{21}$ are Hamiltonian circles in $\Sigma_n$ with four distinct double signs.

Case 2. Suppose there do not exist three consecutive triangles with distinct double signs in \(\mathcal{B}\).   
Notice that $\mathcal{B}$ contains two triangles sharing a common edge but having different double signs, as well as another triangle in $\mathcal{B}$ that is edge-disjoint from these two and has the remaining double sign.
  Without loss of generality, assume the two triangles have double signs \(x\) and \(y\) with vertices \(v_1, v_2, v_3, v_4\), and the other triangle has double sign \(z\) with vertices \(v_1, v_5, v_6\), as shown in the figure below, and then  normalize \(v_5\), so $\sigma(e_{16})=z$.  

\begin{figure}[H]
    \centering
\begin{tikzpicture}[scale=2.5]
  \coordinate (v1) at (-1,0.85);
  \coordinate (v2) at (-1,-0.2);
  \coordinate (v3) at (0.3,-.2);
  \coordinate (v4) at (0.3,0.5);
 \coordinate (v5) at (.3,1);
  \coordinate (v6) at (.3,1.6);
  \draw[thick] (v2) -- (v1) -- (v3) -- cycle; 
  \draw[thick] (v3) -- (v1) -- (v4) -- cycle; 
  \draw[thick] (v2) -- (v1) -- (v4);     
   \draw[thick] (v1) -- (v5) ;  
   \draw[thick] (v1) -- (v6) -- (v5); 

  \node at (-0.5,0.2) {\(x\)};
  \node at (0.1,0.3) {\(y\)};
  \node at (-0.2,1.2) {\(z\)};
  
  \node[above] at (v1) {\(v_1\)};
    \node[below] at (v2) {\(v_2\)};
\node[below] at (v3) {\(v_3\)};
\node[right] at (v4) {\(v_4\)};
\node[right] at (v5) {\(v_5\)};
\node[right] at (v6) {\(v_6\)};
\end{tikzpicture}
\label{three3}
 \caption{}
\end{figure}

Subcase 1. If some edge of $K_4$ on $\{v_1,v_2,v_3,v_4\}$ has double sign $z$, the configuration reduces to Case~1 and yields four distinct Hamiltonian circle double signs as before.

Subcase 2. Suppose no edge in this \(K_4\) has double sign \(z\), i.e., all its edges have double signs \(x\) or \(y\).  
Connect \(v_2\) to \(v_4\) and delete the edge \(v_1 v_3\) from the last figure to obtain the new figure $F$.  
By Lemma \ref{lemma1}, the triangle \(T_{124}\) has double sign \(\alpha\), where \(\alpha \in \{x, y\}\),  
and \(T_{234}\) has double sign \(\beta\), where \(\beta \in \{x, y\} \setminus \{\alpha\}\). There are three isomorphic
types of assignments for $F$, as follows:

\begin{figure}[H]
    \centering
\begin{tikzpicture}[scale=2.5]

\begin{scope}[shift={(0,0)}]
  \coordinate (v1) at (-1,0.85);
  \coordinate (v2) at (-1,-0.2);
  \coordinate (v3) at (0.3,-.2);
  \coordinate (v4) at (0.3,0.5);
  \coordinate (v5) at (0.3,1);
  \coordinate (v6) at (0.3,1.6);

\node at (-0.2,-0.3) {\(\beta\)};
  \node at (-0.2,0.7) {\(\beta\)};
  \node at (0.4,.1) {\(\beta\)};
\node at (-0.4,.2) {\(\beta\)};

  \draw[thick] (v1) -- (v2) -- (v3) ;
  \draw[thick] (v2) -- (v3) -- (v4) ;
  \draw[thick] (v2) -- (v1) -- (v4);     
  \draw[thick] (v1) -- (v5);  
  \draw[thick] (v1) -- (v6) -- (v5); 
 \draw[thick] (v2) -- (v4);
 
  \node at (-1.1,0.2) {\(\alpha\)};
  
  \node at (-0.2,1.4) {\(z\)};
  \node[above] at (v1) {\(v_1\)};
  \node[below] at (v2) {\(v_2\)};
  \node[below] at (v3) {\(v_3\)};
  \node[right] at (v4) {\(v_4\)};
  \node[right] at (v5) {\(v_5\)};
  \node[right] at (v6) {\(v_6\)};
\end{scope}

\begin{scope}[shift={(2,0)}]
   \coordinate (v1) at (-1,0.85);
  \coordinate (v2) at (-1,-0.2);
  \coordinate (v3) at (0.3,-.2);
  \coordinate (v4) at (0.3,0.5);
  \coordinate (v5) at (0.3,1);
  \coordinate (v6) at (0.3,1.6);

\node at (-0.2,-0.3) {\(\alpha\)};
  \node at (-0.2,0.7) {\(\beta\)};
  \node at (0.4,.1) {\(\alpha\)};
\node at (-0.4,.2) {\(\beta\)};

  \draw[thick] (v1) -- (v2) -- (v3) ;
  \draw[thick] (v2) -- (v3) -- (v4) ;
  \draw[thick] (v2) -- (v1) -- (v4);     
  \draw[thick] (v1) -- (v5);  
  \draw[thick] (v1) -- (v6) -- (v5); 
 \draw[thick] (v2) -- (v4);
 
  \node at (-1.1,0.2) {\(\alpha\)};
  
  \node at (-0.2,1.4) {\(z\)};
  \node[above] at (v1) {\(v_1\)};
  \node[below] at (v2) {\(v_2\)};
  \node[below] at (v3) {\(v_3\)};
  \node[right] at (v4) {\(v_4\)};
  \node[right] at (v5) {\(v_5\)};
  \node[right] at (v6) {\(v_6\)};
\end{scope}

\begin{scope}[shift={(4,0)}]
   \coordinate (v1) at (-1,0.85);
  \coordinate (v2) at (-1,-0.2);
  \coordinate (v3) at (0.3,-.2);
  \coordinate (v4) at (0.3,0.5);
  \coordinate (v5) at (0.3,1);
  \coordinate (v6) at (0.3,1.6);

\node at (-0.2,-0.3) {\(\alpha\)};
  \node at (-0.2,0.7) {\(\alpha\)};
  \node at (0.4,.1) {\(\beta\)};
\node at (-0.4,.2) {\(\alpha\)};

  \draw[thick] (v1) -- (v2) -- (v3) ;
  \draw[thick] (v2) -- (v3) -- (v4) ;
  \draw[thick] (v2) -- (v1) -- (v4);     
  \draw[thick] (v1) -- (v5);  
  \draw[thick] (v1) -- (v6) -- (v5); 
 \draw[thick] (v2) -- (v4);
 
  \node at (-1.1,0.2) {\(\alpha\)};
  
  \node at (-0.2,1.4) {\(z\)};
  \node[above] at (v1) {\(v_1\)};
  \node[below] at (v2) {\(v_2\)};
  \node[below] at (v3) {\(v_3\)};
  \node[right] at (v4) {\(v_4\)};
  \node[right] at (v5) {\(v_5\)};
  \node[right] at (v6) {\(v_6\)};
\end{scope}

\end{tikzpicture}

 \caption{}
 \label{three4}
\end{figure}

\vspace{1cm}

Use the left panel of Figure \ref{three4}, we can construct two Hamiltonian circles in \(\Sigma_n \setminus v_5\):  
\[
H_1 := v_6 v_1 v_2 v_4 v_3 v_7 v_8 \ldots v_n, 
\quad 
H_2 := v_6 v_1 v_4 v_2 v_3 v_7 v_8 \ldots v_n.
\]
Denote $J:=v_6 v_1 v_2  v_3 v_7 v_8 \ldots v_n$, assume that $\phi(J)=l$, for some $l \in \mathbb{F}_2^2$. We have
$\phi(H_1) = \phi(J) + \phi(T_{234})
= l+ \beta$, \quad
$\phi(H_2) = \phi(J) + \phi(T_{124})= l+ \alpha.$
Construct Hamiltonian circles in $\Sigma_n$, by using $H_1$ or $H_2$ with some triangles, as follows: 
\[
\begin{aligned}
\phi(H_1^{16})&=\phi(H_1)+\phi(T_{516})=l+\beta+z\\
\phi(H_1^{12})&=\phi(H_1)+\phi(T_{512})=l+\beta+\alpha\\
\phi(H_1^{24})&=\phi(H_1)+\phi(T_{524})=l+\beta+\beta=l\\
\phi(H_2^{16})&=\phi(H_2)+\phi(T_{516})=l+\alpha+z.
\end{aligned}
\]
Since $\alpha, \beta,z$ are distinct, the four values 
\[
l+\beta+z,\quad l+\beta+\alpha,\quad l,\quad l+\alpha+z
\]
are distinct in $\mathbb{F}_2^2$. 
Thus, these yield four Hamiltonian circles with distinct double signs.

The same reasoning applies to the second and third assignment types in Subcase 2.

\smallskip
In all cases, we obtain four Hamiltonian circles in $\Sigma_n$ with distinct double signs. 
\end{proof}

\section{Triangle Double Signs with Four Values}

\begin{definition}\label{def:sigma4-tri-distinct}
A doubly signed complete graph on four vertices is a triple
\(\Sigma_4 := (K_4,\sigma,\mathbb{F}_2^2)\), where
\(\sigma : E(K_4) \to \mathbb{F}_2^2\).
If the four triangles of \(K_4\) have pairwise distinct double signs, we write
\(\Sigma_4^*\).

If, in such a \(\Sigma_4^*\), there are exactly three edges that either share a
common vertex or form a triangle and these three edges have a common double sign
\(x \in \mathbb{F}_2^2\), we write \(\Sigma_4^{**}\). Notice that double signs of six edges in \(\Sigma_4^{**}\) are $x,x,x,y,z,$ and $w$, where $x,y,z, w\in \mathbb{F}_2^2$ are distinct.
\end{definition}

\begin{lemma}\label{lemma14}
    If $\Sigma_4:=(K_4,\sigma,\mathbb{F}_2^2)$ does not contain four distinct double signs of triangles, then its four triangles have double signs $x,x,y,y$ for some $x,y\in\mathbb{F}_2^2$.
\end{lemma}

\begin{lemma}\label{lemma15}
   Normalize $v_1$. If a path of $\Sigma_n\backslash v_1$ contains four distinct double signs of edges, then  $\Sigma_n$ contains all four distinct double signs of Hamiltonian circles.
\end{lemma}

\begin{lemma}\label{lemmac}
If \(\Sigma_n\) contains four distinct  double signs among all triangles and $n>5$, then \(\Sigma_n\) contains all possible double signs of Hamiltonian circles.
\end{lemma}
\begin{proof}

We consider two cases: CASE $\alpha$ and CASE $\beta$.

    CASE $\alpha$. Suppose that \(\Sigma_n\) does not contain $\Sigma_4^*$. Let $x,y,z$ and $w$ be four distinct double signs. Suppose there is no triangles have double signs $x$ in $\mathcal{B}$. Then, there must exist a triangle with double sign $x$ in $\Sigma_n$ but not in  $\mathcal{B}$. Without loss of generality, say this triangle is $T_{234}$. Let $\phi(T_{123})=\alpha$, where $\alpha\in\{y,z,w\}.$ By Lemma \ref{lemma1}, this $K_4$ only contains distinct double signs $x,\alpha$. Observe that $\phi(T_{123})+\phi(T_{234})=\phi(T_{124})+\phi(T_{134})$, that is $\alpha+x=\phi(T_{124})+\phi(T_{134})$, so one of $\phi(T_{124}) $ and $\phi(T_{134})$ has double sign $x$, a contradiction. Thus, $\mathcal{B}$ contains a triangle with double sign $x$. Similarly,
     it contains signs $y,z,w$. 
     
     Pick four triangles $T_{1ab},T_{1cd},T_{1fg}$ and $T_{1hi}$ with distinct double signs from $\mathcal{B}$. Normalize $v_1$. Then,  the edges $e_{ab},e_{cd},e_{fg}$ and $e_{hi}$ will have distinct double signs, say $x,y,z,w$ respectively. 
     Observe that these four edges can not form a length-4 circles , and any three edges of them can not form a length-3 circle. Suppose these four edges form a length-4 circle $v_2v_3v_4v_5$, as shown in the left panel of the Figure \ref{four1}. 
\vspace{-6pt}
\begin{figure}[H]
    \centering 
\begin{tikzpicture}[scale=4, every node/.style={scale=1}]
    \node[blue] (v2) at (0,1) {$v_2$};
    \node[blue] (v3) at (1,1) {$v_3$};
    \node[blue] (v4) at (1,0) {$v_4$};
    \node[blue] (v5) at (0,0) {$v_5$};
    \node[blue] (v1) at (0.25,0.25) {$v_1$};
    \draw[black] (v2) -- (v3);
    \draw[black] (v3) -- (v4);
    \draw[black] (v4) -- (v5);
    \draw[black] (v5) -- (v2);
    \draw[red]  (v2) -- (v4);
    \draw[black] (v1) -- (v2);
    \draw[black] (v1) -- (v3);
    \draw[black] (v1) -- (v4);
    \draw[black] (v1) -- (v5);
    \draw[black, thick] (v5) -- node[left, orange] {$x$} (v2);
    \draw[black, thick] (v3) -- node[above, orange] {$y$} (v2);
    \draw[black, thick] (v3) -- node[left, orange] {$z$} (v4);
    \draw[black, thick] (v5) -- node[below, orange] {$w$} (v4);

    \begin{scope}[xshift=1.6cm]
        \node[blue] (v2b) at (0,1) {$v_2$};
       
        \node[blue] (v4b) at (1,0) {$v_4$};
        \node[blue] (v5b) at (0,0) {$v_5$};
        \node[blue] (v1b) at (0.25,0.25) {$v_1$};
         \draw[black] (v4b) -- (v5b);
        \draw[black] (v5b) -- (v2b);
        \draw[black]  (v2b) -- (v4b);
        \draw[black] (v1b) -- (v2b);
        \draw[black] (v1b) -- (v4b);
        \draw[black] (v1b) -- (v5b);
     \draw[black, thick] (v5b) -- node[left, orange] {$x$} (v2b);
        \draw[black, thick] (v4b) -- node[above, orange] {$y$} (v2b);
       \draw[black, thick] (v5b) -- node[below, orange] {$w$} (v4b);
    \end{scope}
\end{tikzpicture}

 \caption{}
 \label{four1}
\end{figure}
\vspace{-6pt}

Recall that all the edges incident with $v_1$ have double sign $e$.  If the edge $v_2v_4$ has double sign $x$ or $w$, then the $K_4$ subgraph induced by $\{v_1,v_2,v_3,v_4\}$ will contain four distinct double signs of triangles, a contradiction. Thus, the edge $v_2v_4$ must have double sign $z$ or $y$. But then, the $K_4$ subgraph induces by $\{v_1,v_2,v_5,v_4\}$ will contain four distinct double signs fo triangle, a contradiction again.  Hence, these four edges cannot form a length-4 circle. Suppose three of four edges form a length-3 circle, Without loss of generality, assume their signs are $x,y,w$, as shown on the right panel of the Figure \ref{four1}. However, this $K_4$ contains four distinct double signs of triangles, a contradiction.

From the last two paragraphs, we obtain that these four edges must form a forest. We consider four cases. In each case, it suffices to exhibit a path that uses all four edges with distinct double signs; then, by Lemma~\ref{lemma15}, we conclude that $\Sigma_n$ contains all four possible double signs of Hamiltonian circles.

Case 1. Suppose these four edges share a common vertex $v_2$, as shown in the figure below. 

\begin{figure}[H]
    \centering
\begin{tikzpicture}[scale=3, every node/.style={font=\small}]
  \coordinate (v4) at (0,1);
  \coordinate (v5) at (2,1);
  \coordinate (v3) at (0,0);
  \coordinate (v6) at (2,0);
  \coordinate (v1) at (1,1);   
  \coordinate (v2) at (1,0);   

  \node[above left]  at (v4) {$v_4$};
  \node[above right] at (v5) {$v_5$};
  \node[below left]  at (v3) {$v_3$};
  \node[below right] at (v6) {$v_6$};
  \node[above]       at (v1) {$v_1$};
  \node[below]       at (v2) {$v_2$};

  \draw[black] (v4) -- (v1) -- (v5);           
  \draw[black] (v3) -- (v2) -- (v6);           
  \draw[black] (v1) -- (v2);                   

  \draw[red, very thick] (v4) -- (v3);
  \draw[red, very thick] (v5) -- (v6);

  \draw[red, very thick] (v3) -- node[below, text=red] {$x$} (v2);
  \draw[red, very thick] (v2) -- node[below, text=red] {$w$} (v6);

  \draw[red, very thick] (v4) -- node[pos=0.55, right, text=red] {$y$} (v2);
  \draw[red, very thick] (v5) -- node[pos=0.45, right, text=red] {$z$} (v2);
\draw[black,  thick] (v3) -- node[pos=0.45, right, text=red] {} (v1);
\draw[black,  thick] (v6) -- node[pos=0.45, right, text=red] {} (v1);

\end{tikzpicture}
\label{four2}
 \caption{}
\end{figure}

\vspace{1cm}

Observe the induced subgraph $K_4$ by $\{v_1,v_2,v_3,v_4\}$. Since it can not contain four distinct double signs of triangles, the edge $v_3v_4$ has double sign  either $x$ or $y$. Thus, either the path $p_1:=v_3v_4v_2$ or the path  $p_2:=v_4v_3v_2$ has double sign $x+y$. Similarly, in induced subgraph $K_4$ by $\{v_1,v_2,v_6,v_5\}$, either the path $p_3:=v_2v_5v_6$ or the path $p_4:=v_2v_6v_5$ has double sign $w+z$. Choose one of the path from $p_1,p_2$, choose another path from $p_3,p_4$, joint them by $v_2$ such that the four edges have distinct double signs. 

Case 2. Suppose three edges share a common vertex, and one of these three edge share another vertex with the last edge, as shown in the Figure \ref{four3}. 

\begin{figure}[H]
    \centering
\begin{tikzpicture}[scale=3, every node/.style={font=\small}]
  \coordinate (v4) at (0,1);
  \coordinate (v5) at (3,1);
  \coordinate (v3) at (0,0);
  \coordinate (v6) at (2,0);
  \coordinate (v1) at (1,1);   
  \coordinate (v2) at (1,0);   

  \node[above left]  at (v4) {$v_4$};
  \node[above right] at (v5) {$v_5$};
  \node[below left]  at (v3) {$v_3$};
  \node[below right] at (v6) {$v_6$};
  \node[above]       at (v1) {$v_1$};
  \node[below]       at (v2) {$v_2$};

  \draw[black] (v4) -- (v1) -- (v5);           
  \draw[black] (v3) -- (v2) -- (v6);           
  \draw[black] (v1) -- (v2);                   

  \draw[green, very thick] (v4) -- (v3);
  \draw[red, very thick] (v5) -- (v6);

  \draw[red, very thick] (v3) -- node[below, text=red] {$x$} (v2);
  \draw[red, very thick] (v2) -- node[below, text=red] {$w$} (v6);

  \draw[red, very thick] (v4) -- node[pos=0.55, right, text=red] {$y$} (v2);
  \draw[red, very thick] (v5) -- node[pos=0.45, right, text=red] {$z$} (v6);
\draw[black,  thick] (v3) -- node[pos=0.45, right, text=red] {} (v1);
\draw[black,  thick] (v6) -- node[pos=0.45, right, text=red] {} (v1);

\end{tikzpicture}

 \caption{}
 \label{four3}
\end{figure}

\vspace{1cm}

Observe that the edges $v_4v_2,v_3v_2,v_2v_6$ share a common vertex and have distinct double signs $y,x,w$. Since the induced subgraph $K_4$ by $\{v_1,v_2,v_3,v_4\}$ does not contain four double signs of triangles, $\sigma(v_3v_4)=x$ or $y$. Thus, either the path $v_4v_3v_2v_6v_5$ or the path $v_3v_4v_2v_6v_5$ contains four double signs of edges.

Case 3. Suppose three edges share a common vertex, and they do not share  vertices  with last edge. Use the similar reasoning as in case 2 find a path with length-3 that contains three distinct double signs of edges, then connect this path with the last edge to form a new path. Thus, this new path contains four distinct double signs of edges.

     Case 4.  Suppose that, among these four edges, no vertex is incident with more than one of the other edges. If the four edges already form a length-4 path, we are done. Otherwise, they decompose into disjoint paths; connect these paths to obtain a longer path that uses all four edges, which therefore contains all four edge double signs.

 \vspace{1cm}

   CASE $\beta$. Suppose that \(\Sigma_n\)  contains $\Sigma_4^*$. We need the following lemmas to prove that \(\Sigma_n\)  contains all possible double signs of Hamiltonian circles.
    \end{proof}

\begin{lemma}\label{11}
Let \(y_1,y_2\in\mathbb F_2^2\) be distinct nonzero elements, and let \(z_1,z_2\in\mathbb F_2^2\) be distinct elements such that exactly one of \(y_1,y_2\) equals one of \(z_1,z_2\).
Then \(\{\,y_i+z_j: i,j\in\{1,2\}\,\}=\mathbb F_2^2\).

\end{lemma}

\begin{proof}
Let $y_1, y_2, y_3$ be the three distinct non-identity elements of $\mathbb{F}_2^2$.  
Note that $y_3 = y_1+ y_2$.  
Without loss of generality, assume $y_1 = z_1$.  
Then $z_2$ is either $e$ or $y_3$.

\textbf{Case 1:} $z_2 = e$.  
Then $y_1 +z_2 = y_1$, $y_2 +z_2 = y_2$, $y_1+ z_1 = e$, and $y_2+z_1 = y_2+ y_1 = y_3$.  
Thus $\{y_i +z_j\}_{i,j} = \mathbb{F}_2^2$.

\textbf{Case 2:} $z_2 = y_3$.  
Then $y_1+z_2 = y_2$, $y_2+ z_2 = y_1$, $y_1+ z_1 = e$, and $y_2 +z_1 = y_3$.  
Thus $\{y_i +z_j\}_{i,j} = \mathbb{F}_2^2$.

In either case, the conclusion follows.
\end{proof}

\begin{lemma}\label{12} 
Let $y_1, y_2 \in \mathbb{F}_2^2$ be distinct non-identity elements, and let $z_1, z_2 \in \mathbb{F}_2^2$ be distinct elements such that 
\[
\{z_1, z_2\} = \{y_1, y_2\} \quad \text{or} \quad \{z_1, z_2\} = \mathbb{F}_2^2 \setminus \{y_1, y_2\}.
\]
Then $\{y_i +z_j\}_{i,j}$ is equal to the multiset $\{x, x, y, y\}$ for some $x, y \in \mathbb{F}_2^2$.
\end{lemma}

\begin{proof}
First, assume $\{z_1, z_2\} = \{y_1, y_2\}$.  
Without loss of generality, let $z_1 = y_1$ and $z_2 = y_2$. Then
\[
\{y_i +z_j : 1 \le i, j \le 2\} = \{y_1+y_1,\, y_1+y_2,\, y_2+y_1,\, y_2+y_2\}.
\]
Observe that $y_1+y_1 = y_2+y_2 = e$ and $y_1+y_2 = y_2+y_1$.  
Thus the multiset is $\{x, x, y, y\}$ for some $x, y \in \mathbb{F}_2^2$.

Now assume $\{z_1, z_2\} = \mathbb{F}_2^2 \setminus \{y_1, y_2\}$.  
Since $\mathbb{F}_2^2$ has three distinct non-identity elements $y_1, y_2, y_3$ with $y_3 = y_1+y_2$, we have $\{z_1, z_2\} = \{y_3, e\}$.  
The set of products is then
\[
\{y_1+y_3,\, y_1+e,\, y_2+y_3,\, y_2+e\}.
\]
Note that $y_1+y_3 = y_2+e$ and $y_1+e = y_2+y_3$, so again the multiset is $\{x, x, y, y\}$ for some $x, y \in \mathbb{F}_2^2$.
\end{proof}

Let $e$ be the identity element of $\mathbb{F}_2^2$, and let $a,b,c$ be three distinct non-identity elements of $\mathbb{F}_2^2$. 

\textbf{Convention.}
Whenever we refer to \(\Sigma_4^*\), we fix a labeled copy with vertex set
\(\{v_1,v_2,v_3,v_4\}\) and triangle double signs
\[
\sigma(T_{123})=a,\qquad \sigma(T_{124})=b,\qquad
\sigma(T_{134})=c,\qquad \sigma(T_{234})=e .
\]
This is without loss of generality: any \(\Sigma_4^*\) is isomorphic to this one
by relabeling the vertices, and our arguments depend only on the fact that the
four triangle signs are distinct.

 There are 12 distinct Hamiltonian paths in $\Sigma_4^*$ that split into three groups of four: Group 1 consists of Hamiltonian paths from $v_1$ to $v_2$ or from $v_3$ to $v_4$; Group 2 consists of Hamiltonian paths from $v_1$ to $v_4$ or from $v_2$ to $v_3$; and Group 3 consists of Hamiltonian paths from $v_1$ to $v_3$ or $v_2$ to $v_4$.

\begin{lemma}[Key Lemma]\label{key}
In $\Sigma_4^*$, the number of Hamiltonian paths with double sign $g$ is even for each $g\in \mathbb{F}_2^2$. 

\end{lemma}
\begin{proof}

Let $v_1,v_2,v_3,v_4$ be the vertices of $\Sigma_4^*$ as shown in Figure below.  

\vspace{1cm}

\begin{figure}[H]
    \centering
\begin{tikzpicture}[scale=2, every node/.style={scale=1}]
    \node[circle, draw=none, fill=purple!20, inner sep=2pt, label=left:{$v_1$}] (v1) at (0,1) {};
    \node[circle, draw=none, fill=purple!20, inner sep=2pt, label=right:{$v_2$}] (v2) at (2,1) {};
    \node[circle, draw=none, fill=purple!20, inner sep=2pt, label=right:{$v_3$}] (v3) at (2,0) {};
    \node[circle, draw=none, fill=purple!20, inner sep=2pt, label=left:{$v_4$}] (v4) at (0,0) {};
    
    \draw (v1) -- node[above] {$e_{12}$} (v2);
    \draw (v2) -- node[right] {$e_{23}$} (v3);
    \draw (v3) -- node[below] {$e_{34}$} (v4);
    \draw (v4) -- node[left] {$e_{14}$} (v1);
    \draw (v1) -- (v3);
    \draw (v2) -- (v4);
\end{tikzpicture}
\label{key2}
 \caption{}
\end{figure}

\vspace{1cm}

Denote three Hamiltonian circles in $\Sigma_4^*$ as
$C_1 ($The circle $v_1v_2v_3v_4$), $C_2 (v_1v_2v_4v_3)$, $C_3 (v_1v_3v_2v_4)$. The double signs of  $C_1$   is $\sigma(C_1)=\sigma(T_{123})+\sigma(T_{134})=a+c=b$. Similarly, $\sigma(C_2)=a$  and $\sigma(C_3)=c$.  There are three pairs of disjoint edges: $e_{12}$ and $e_{34}$ (Both edges in $C_1,C_2$), $e_{14}$ and $e_{23}$ (Both edges in $C_1,C_3$), $e_{13}$ and $e_{24}$ (Both edges in $C_3,C_2$).  Consider disjoint edges $e_{12}$ and $e_{34}$.
 There are two Hamiltonian paths $v_1v_4v_3v_2$ and $v_1v_3v_4v_2$ starting at $v_1$ and ending at $v_2$. Similarly, there are another two Hamiltonian paths $v_4v_1v_2v_3$ and $v_4v_2v_1v_3$ starting at $v_4$ and ending at $v_3$. These four paths are in group 1. 
The double signs of 
 the path $v_1v_4v_3v_2$ is $\sigma(v_1v_4v_3v_2)=\sigma(C_1)+\sigma(e_{12})$. Likewise, $\sigma(v_1v_3v_4v_2)=\sigma(C_2)+\sigma(e_{12})$, $\sigma(v_4v_1v_2v_3)=\sigma(C_1)+\sigma(e_{34})$,  and
$\sigma(v_4v_2v_1v_3)=\sigma(C_2)+\sigma(e_{34})$.

Notice that $\sigma(e_{12})+\sigma(e_{34})$ has four possible double signs $a,b,c,$ and $e$. We will consider them in four cases:

Case 1.  $\sigma(e_{12})+\sigma(e_{34})=a$.
Since $\sigma(C_1)=\sigma(e_{12})+\sigma(e_{34})+\sigma(e_{23})+\sigma(e_{14})=b$, 
it will force $\sigma(e_{23})+\sigma(e_{14})=c$. Since $\sigma(C_2)=a$, it forces $\sigma(e_{13})+\sigma(e_{24})=e$. (see the table 1 below).
Observe that edges $e_{12}$ and $e_{34}$ are in $C_1$ and $C_2$.
Since $\sigma(e_{12})+\sigma(e_{34})$ has the same double sign to one of $\sigma(C_1)(=b)$ or $\sigma(C_2)(=a)$, we say that the value $\sigma(e_{12})+\sigma(e_{34})$ \textcolor{blue}{agrees} on one of $\sigma(C_1)$ or $\sigma(C_2)$ (see the table 1 below). 
Since $\sigma(e_{12})+\sigma(e_{34})=a$, one possibility is that $\sigma(e_{12})=a$ and $\sigma(e_{34})=e$ (or  $\sigma(e_{34})=a$ and $\sigma(e_{12})=e$, that will lead to the same result); alternatively, we could have $\sigma(e_{12})=b$ and $\sigma(e_{34})=c$ (or  $\sigma(e_{34})=b$ and $\sigma(e_{12})=c$, that will provide the same result).  Now,  we want to apply Lemma \ref{11}, by setting $y_1=\sigma(C_1), y_2=\sigma(C_2), z_1=\sigma(e_{12}), z_2=\sigma(e_{34})$. Then, 
by Lemma \ref{11}, the double signs of the four Hamiltonian paths $v_1v_4v_3v_2, v_1v_3v_4v_2, v_4v_1v_2v_3, v_4v_2v_1v_3$ will be all different, i.e. double signs of $a,b,c,e$. 

Next,  $\sigma(e_{14})+\sigma(e_{23})=\sigma(C_1)+\sigma(e_{12})+\sigma(e_{34})=b+a=c$, this double sign agrees with $\sigma(C_1)$ or $\sigma(C_3)$. There are four possible assignments for $\sigma(e_{14})$ and $\sigma(e_{23})$: 

 $\sigma(e_{14})=a, \sigma(e_{23})=b$ ;
 
$\sigma(e_{14})=b, \sigma(e_{23})=a$;

 $\sigma(e_{14})=c, \sigma(e_{23})=e$;
 
 $\sigma(e_{14})=e, \sigma(e_{23})=c$.

\vspace{0.5cm}
Set $y_1=\sigma(C_1)$, $y_2=\sigma(C_3), z_1=\sigma(e_{14})$, and $z_2=\sigma(e_{23})$. By Lemma \ref{11}, four Hamiltonian paths in group 2 will have distinct double signs.

Finally, we want to check the double signs of four Hamiltonian paths in group 3. 
$\sigma(e_{13})+\sigma(e_{24})=\sigma(C_2)+\sigma(e_{12})+\sigma(e_{34})=a+a=e$, this double sign is not the same as $\sigma(C_2)$ or $\sigma(C_3)$, we say it does  \textcolor{blue}{not agrees} with $\sigma(C_2)$ or $\sigma(C_3)$. 
Then, $\sigma(e_{13})=\sigma(e_{24})=k$, for some $k\in \mathbb{F}_2^2$. The double signs of the four Hamiltonian paths in group 3 are  $\sigma(C_2)+\sigma(e_{13})=a+k$, $\sigma(C_2)+\sigma(e_{24})=a+k, \sigma(C_3)+\sigma(e_{13})=c+k$, and $\sigma(C_3)+\sigma(e_{24})=c+k$.  

(\textbf{Note.} The reader can verify that, in certain cases marked ``not agree,'' one may apply Lemma~\ref{12}.)

Observe that the number of the 12 distinct Hamiltonian paths above with double sign $g$ is
even for each $g\in \mathbb{F}_2^2$.

For Cases~2, 3, and~4, we summarize the results in Table~1 below.  
When the entry reads “agree,” the four Hamiltonian paths in that group have distinct double signs.  
When the entry reads “not agree,” the four Hamiltonian paths in that group have double signs \(s,s,t,t\),  
where \(s,t \in \mathbb{F}_2^2\) are distinct.  
In each case, the number of Hamiltonian paths with a given double sign \(g\) is even for every \(g \in \mathbb{F}_2^2\).

\begin{center}
\begin{tabular}{||c| c| c ||} 
 \hline
Given values & responded values & responded  values  \\ [0.5ex] 
 \hline
 $\sigma(e_{12})+\sigma(e_{34})$ & $\sigma(e_{14})+\sigma(e_{23})$ & $\sigma(e_{13})+\sigma(e_{24})$  \\ 
 
 \hline\hline
 Case 1:  a (agree)& c (agree)& e (not agree)\\
 \hline
 Case 2: e (not agree)& b (agree)& a (agree) \\
 \hline
Case 3: b (agree)& e (not agree)& c (agree)\\
 \hline
Case 4:  c (not agree)& a (not agree)& b (not agree)\\ [1ex] 
 \hline \hline
 $e_{12},e_{34}$ are on $C_1,C_2$.  & $e_{14}$,$e_{23}$ are on $C_1,C_3$.& $e_{13},e_{24}$ are on $C_2,C_3$.\\ [1ex] 
  $\sigma(C_1)=b$ and $\sigma(C_2)=a$ & $\sigma(C_1)=b$ and $\sigma(C_3)=c$& $\sigma(C_2)=a$ and $\sigma(C_3)=c$\\ [1ex]

  \hline
\end{tabular}
\end{center}
\begin{center}
    Table 1
\end{center}
\end{proof}

\vspace{2cm}

In $\Sigma_4^*$, for each $i\in\{1,2,3,4\}$, let $\mathscr{P}_i$ denote the set of the six Hamiltonian paths that start at $v_i$. We denote by $\sigma(\mathscr{P}_i)$ the multiset of double signs corresponding to these Hamiltonian paths.

\begin{lemma}\label{lemma4}
For each \(i\in\{1,2,3,4\}\), the multiset of double signs of the six Hamiltonian paths
\(\mathscr{P}_i\) starting at \(v_i\) is either
\(\{p,p,q,q,s,s\}\) or \(\{p,q,s,t,t,t\}\) for some distinct
\(p,q,s,t\in\mathbb F_2^2\).
\end{lemma}

\begin{proof}
By symmetry it suffices to take \(i=1\).
Assume (WLOG for \(\Sigma_4^*\)) that
\(\sigma(T_{123})=a\), \(\sigma(T_{124})=b\), \(\sigma(T_{134})=c\), and
\(\sigma(T_{234})=e\) (the identity). 

Let \(x=\sigma(e_{12})\), \(y=\sigma(e_{13})\), \(z=\sigma(e_{14})\).
The six Hamiltonian paths from \(v_1\) arise by deleting one of the three
edges \(e_{12},e_{13},e_{14}\) from one of the Hamiltonian cycles
\(C_1=v_1v_2v_3v_4\), \(C_2=v_1v_2v_4v_3\), \(C_3=v_1v_3v_2v_4\).
Since \(\sigma(C_1)=b\), \(\sigma(C_2)=a\), \(\sigma(C_3)=c\) and $\sigma(C_i\backslash e_{jk})=\sigma(C_i)+\sigma(e_{jk})$,
the six paths   \(C_{2}\setminus e_{12}\), \(C_{1}\setminus e_{12}\),
\(C_{2}\setminus e_{13}\),  \(C_{3}\setminus e_{13}\), $C_1\backslash e_{14},$ and $ C_3\backslash e_{14}$ 
have signs 
\[
\underbrace{a+x,\ b+x}_{\text{omit }e_{12}},\qquad
\underbrace{a+y,\ c+y}_{\text{omit }e_{13}},\qquad
\underbrace{b+z,\ c+z}_{\text{omit }e_{14}},
\] respectively.

(1) Within each brace, the two values are distinct.
Indeed,
\((a+x)+(b+x)=a+b=c\neq e\), and \((b+z)+(c+z)=b+c=a\neq e\), hence
\(a+x\ne b+x\) and \(b+z\ne c+z\). Similarly \((a+y)+(c+y)=a+c=b\ne e\), so
\(a+y\ne c+y\).

(2) Between each pair of braces, four double signs cannot be all distinct. 
For example, between \(\{a+x,b+x\}\) and \(\{b+z,c+z\}\):  Suppose all of them are different. $a+x\neq b+z$ or $c+z$ implies $x\neq c+z$ or $b+z$. $b+x\neq b+z$ or $c+z$ implies $x\neq e+z$ or $a+z$. It would force 
$x$ to avoid all four elements of $\mathbb{F}_2^2$, a contradiction. Thus, $a+x$, $b+x$, $b+z$ and $c+z$ cannot be all different.

(3) Between each pair of braces, exactly one element from one brace equals an element from the other.
 Suppose there are two pairs of them have the same double signs. Assume 
     $a+x=b+z$, $b+x=c+z$. This implies $x=c+z$ and $x=a+z$, impossible. Similary, $a+x=c+z$, $b+x=b+z$ cannot happen. Hence, there are exactly two of $a+x$, $b+x$, $b+z$ and $c+z$ are the same. 

\vspace{0.5cm}

Finally, we can analyze the double signs of the six paths starting at $v_1$.

Case 1. Three of $a+y,c+y$, $a+x$, $b+x$, $b+z$ and $c+z$ are the same. Then, remaining three of them are all different. In this case, we have all four double signs $a,b,c,e$ with exactly one of them appearing three times. 

Case 2. No three elements of $a+y,c+y$, $a+x$, $b+x$, $b+z$ and $c+z$  share the same double sign. Consequently, there are
precisely three pairs of equal values, and the value of each pair is
different from the others.

\vspace{0.5cm}

Therefore, \(\sigma(\mathscr P_1)\) is either \(\{p,p,q,q,s,s\}\) or
\(\{p,q,s,t,t,t\}\) with \(p,q,s,t\) distinct. By symmetry the same holds for
\(i=2,3,4\).
\end{proof}

\begin{lemma}\label{same}
Let $v_1,v_2,v_3,v_4$ be four vertices of   $\Sigma_4^*$ and $\mathbb{F}_2^2:=\{g,r,s,t\}$.
 Then, the following are equivalent:
\vspace{1cm}

(1) There exist two distinct vertices \(v_i\) and \(v_j\) in \(\Sigma_4^*\) such that
\[
\sigma(\mathscr{P}_i) = \sigma(\mathscr{P}_j),
\]
and this common multiset contains exactly two copies of each of the double signs \(r\), \(s\), and \(t\), where  $r,s,t\in \mathbb{F}_2^2$ are distinct and  $i,j\in [4]$ are distinct.

(2) \(\Sigma_4^*\) contains exactly three edges that either share a common vertex or form a triangle with double sign \(g\), where $g \in \mathbb{F}_2^2 \setminus \{r,s,t\}$.

(3) The graph $\Sigma_4^*$ does not contain a Hamiltonian path with a double sign $g$, where $g \in \mathbb{F}_2^2 \setminus \{r,s,t\}$.

\end{lemma}

\vspace{1cm}

\begin{figure}[H]
    \centering
\begin{tikzpicture}[scale=1.5, every node/.style={font=\small}]
  \coordinate (v1) at (0,1.7);
  \coordinate (v2) at (3,1.7);
  \coordinate (v3) at (3,0);
  \coordinate (v4) at (0,0);

  \draw[thick] (v1)--(v2)--(v3)--(v4)--cycle;
  \draw[thick] (v1)--(v3);
  \draw[thick] (v2)--(v4);
  
\draw[black,thick] (v2) -- node[pos=0.55, above, text=black] {$x$} (v1);
\draw[black,thick] (v1) -- node[pos=0.35, above, text=black] {$y$} (v3);
\draw[black,thick] (v1) -- node[pos=0.55, left, text=black] {$z$} (v4);
\node[left]  at (v4) {$v_4$};
  
  \node[right]  at (v3) {$v_3$};
 
  \node[above]       at (v1) {$v_1$};
  \node[above]       at (v2) {$v_2$};
\end{tikzpicture}

 \caption{}
 \label{xyz}
\end{figure}

\vspace{1cm}

\begin{proof} First, we show (1) implies (2).

Recall that $\sigma(T_{123})=a$, $\sigma(T_{134})=c$, $\sigma(T_{234})=e$ and $\sigma(T_{124})=b$ .
 Without loss of generality, assume that the vertices $v_1,v_2$ satisfy $\sigma(\mathscr{P}_1)=\sigma(\mathscr{P}_2)$.  Label some unknown double signs of the edges incident to $v_1$ as $x,y,z$ by follows:  $\sigma(e_{12})=x, \sigma(e_{13})=y, \sigma(e_{14})=z$ (See Figure \ref{xyz}). Then, $\sigma(e_{23})=\sigma(\triangle v_1v_2v_3)+\sigma(e_{12})+\sigma(e_{13})=a+x+y$ and $\sigma(e_{24})=\sigma(\triangle v_1v_2v_4)+\sigma(e_{12})+\sigma(e_{14})=b+x+z$. Assume that no Hamiltonian path starting from either \(v_1\) or \(v_2\) carries double sign \(a\) (i.e. $g=a$). This implies $\sigma(C_1)+x\neq a$ and $\sigma(C_2)+x\neq a$.   Since $\sigma(C_1)=b$ and  $\sigma(C_2)=a$, we have $x\neq c$ or $e$. 
 This forces $x=a$ or $b$. Similarly, $y=a$ or $c$, $z=a$ or $e$.  Next, we compute the double signs of six paths starting at $v_1$, as well as $v_2$. For convenience, denote these paths by $p_{i1},p_{i2},...,p_{i6}$. One obtains the following table of double signs:

\begin{center}
\begin{tabular}{||c| c| c |c| c |c |c||} 
 \hline
 Starting at & $p_{i1}$ & $p_{i2}$ & $p_{i3}$ & $p_{i4}$ & $p_{i5}$ & $p_{i6}$\\ [0.5ex] 
 \hline\hline
 $v_1$ & $b+x$ & $a+x$&$a+y$ &$c+y$& $b+z$&$c+z$ \\ 
 \hline
 $v_2$ & $b+x$ & $a+x$ & $c+x+y$&$b+x+y$&$c+x+z$&$a+x+z$ \\
 [1ex] 
 \hline
\end{tabular}\\
\end{center}
\begin{center}

 Table 2: double signs of six Hamiltonian paths start at $v_1$, and another six at $v_2$.
\end{center}
We now distinguish two cases based on the double sign of the edge $e_{12}$ (Since this edge is between $v_1$ and $v_2$, we will consider two cases based on this edge), which is the value of  
$x$. 
Case 1. 
 Suppose $x=b$. Then the double signs in the table 2 simplify to:
\begin{center}
\begin{tabular}{||c| c| c |c| c |c |c||} 
 \hline
 Starting at & $p_{i1}$ & $p_{i2}$ & $p_{i3}$ &$p_{i4}$&$p_{i5}$&$p_{i6}$\\ [0.5ex] 
 \hline\hline
 $v_1$ & $e$ & $c$&$a+y$ &$c+y$& $b+z$&$c+z$ \\ 
 \hline
$ v_2$ & $e$ & $c$ & $a+y$&$e+y$&$a+z$&$c+z$ \\
 [1ex] 
 \hline
 
\end{tabular}\\

\end{center}
\begin{center}
   Table 3: Plug in $x=b$ into Table 2 to obtain this table.
\end{center}

Since $\sigma(\mathscr{P}_1)=\sigma(\mathscr{P}_2)$, we have $e+y=b+z$. It then follows that $y+z=b$. We also have $y=a$ or $c$, $z=a$ or $e$. This implies $y=c,z=a$. Calculate the remaining edges, we obtain the graph: 

\vspace{1cm}

\begin{figure}[h]
    \centering
\begin{tikzpicture}[scale=1.5, every node/.style={font=\small}]
  \coordinate (v1) at (0,1.7);
  \coordinate (v2) at (3,1.7);
  \coordinate (v3) at (3,0);
  \coordinate (v4) at (0,0);

  \draw[thick] (v1)--(v2)--(v3)--(v4)--cycle;
  \draw[thick] (v1)--(v3);
  \draw[thick] (v2)--(v4);
  
\draw[black,thick] (v2) -- node[pos=0.55, above, text=black] {$b$} (v1);
\draw[black,thick] (v1) -- node[pos=0.35, above, text=black] {$c$} (v3);
\draw[black,thick] (v1) -- node[pos=0.55, left, text=black] {$a$} (v4);
\node[left]  at (v4) {$v_4$};
\draw[black,thick] (v2) -- node[pos=0.55, right, text=black] {$e$} (v3);
\draw[black,thick] (v4) -- node[pos=0.35, above, text=black] {$a$} (v3);
\draw[black,thick] (v2) -- node[pos=0.55, left, text=black] {$a$} (v4);
\node[left]  at (v4) {$v_4$};

  \node[right]  at (v3) {$v_3$};
 
  \node[above]       at (v1) {$v_1$};
  \node[above]       at (v2) {$v_2$};
\end{tikzpicture}
\caption{Three edges with double sign $a$ that share a vertex.}
    \label{all a1}
\end{figure}
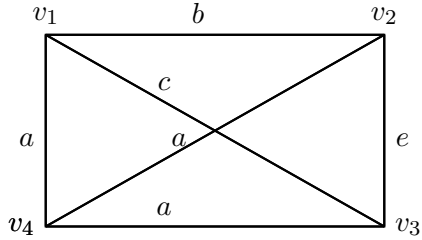
\vspace{1cm}

Thus, $\Sigma_4^*$ contains exactly three edges with double sign 
$a$ that  share a vertex. 

Case 2. 
Suppose instead that $x=a$. Then, the table 2 becomes
\begin{center}
\begin{tabular}{||c| c| c |c| c |c |c||} 
 \hline
 Starting at & $p_{i1}$ & $p_{i2}$ & $p_{i3}$ &$p_{i4}$&$p_{i5}$&$p_{i6}$\\ [0.5ex] 
 \hline\hline
 $v_1$ & $c$ & $e$&$a+y$ &$c+y$& $b+z$&$c+z$ \\ 
 \hline
$ v_2$ & $c$ & $e$ & $b+y$&$c+y$&$b+z$&$e+z$ \\
 [1ex] 
 \hline
\end{tabular}\\
\end{center}
\begin{center}
Table 4: Plug in $x=a$ into Table 2 to obtain this table.
\end{center}

Since $\sigma(\mathscr{P}_1)=\sigma(\mathscr{P}_2)$, we have $a+y=e+z$. Thus, $y+z=a$. Together with $y=a$ or $c$, $z=a$ or $e$, this implies $y=a,z=e$.  Calculate the remaining edges, we obtain the graph: 

\vspace{1cm}
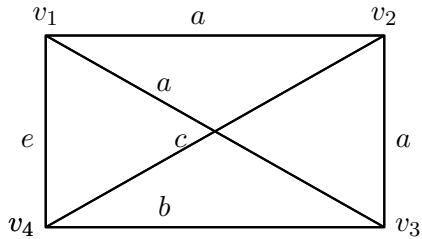
\begin{figure}[h]
    \centering
\begin{tikzpicture}[scale=1.5, every node/.style={font=\small}]
  \coordinate (v1) at (0,1.7);
  \coordinate (v2) at (3,1.7);
  \coordinate (v3) at (3,0);
  \coordinate (v4) at (0,0);

  \draw[thick] (v1)--(v2)--(v3)--(v4)--cycle;
  \draw[thick] (v1)--(v3);
  \draw[thick] (v2)--(v4);
  
\draw[black,thick] (v2) -- node[pos=0.55, above, text=black] {$a$} (v1);
\draw[black,thick] (v1) -- node[pos=0.35, above, text=black] {$a$} (v3);
\draw[black,thick] (v1) -- node[pos=0.55, left, text=black] {$e$} (v4);
\node[left]  at (v4) {$v_4$};
\draw[black,thick] (v2) -- node[pos=0.55, right, text=black] {$a$} (v3);
\draw[black,thick] (v4) -- node[pos=0.35, above, text=black] {$b$} (v3);
\draw[black,thick] (v2) -- node[pos=0.55, left, text=black] {$c$} (v4);
\node[left]  at (v4) {$v_4$};

  \node[right]  at (v3) {$v_3$};
 
  \node[above]       at (v1) {$v_1$};
  \node[above]       at (v2) {$v_2$};
\end{tikzpicture}
\caption{Three edges with double sign $a$ form a triangle.}
    \label{three a}
\end{figure}
\vspace{1cm}

Thus, $\Sigma_4^*$ contains exactly three edges with double sign 
$a$ that  form a triangle.  

The arguments in both cases rely on the assumption that no path have double sign $a$. Next, let $w\in\{b,c,e\}$. Suppose no path have double sign $w$, with the same argument above we will obtain the same result that $\Sigma_4^*$ contains exactly three edges with double sign 
$w$ that share a vertex or form a triangle. 

For any two distinct vertices $v_i,v_j$ of a graph $\Sigma_4^*$, under the condition $\sigma(\mathscr{P}_i)=\sigma(\mathscr{P}_j)$,  use the same reasoning above, it will turn out that  $\Sigma_4^*$ contains exactly three edges carrying the same double sign either share a common vertex or form a triangle. As required.

Next, show (2) implies (3). 

Suppose \(\Sigma_4^*\) contains exactly three edges that either share a common vertex or form a triangle with double sign \(g \in \mathbb{F}_2^2\). Then, remaining three edges will have double signs $r,s$ and $t$.
Partition the double signs of six edges of \(\Sigma_4^*\) into two sets: $A=\{g,g,g\}$ and $B=\{r,s,t\}$. Then, any Hamiltonian path $p$ in \(\Sigma_4^*\) will have double sign: $\sigma(p)=g+x+y$, where $ x\in B, y\in A\cup B\backslash x$. If $y\in A$, then  $\sigma(p)=g+x+g=x$. Thus, the  double sign of this path is not $g$. If $y\in B\backslash x$, then $\sigma(p)=g+x+\alpha=\beta$, where distinct elements $\alpha,\beta\in B\backslash x$. Thus, the  double sign of $p$ is not $g$. 

Finally, we show (3) implies (1).

Suppose the graph $\Sigma_4^*$ does not contain a Hamiltonian path with a double sign $g$, where $g \in \mathbb{F}_2^2 \setminus \{r,s,t\}$. Then, by Lemma \ref{lemma4}, we have  $\sigma(\mathscr{P}_i)=\sigma(\mathscr{P}_j)=\{r,r,s,s,t,t\}$ for some distinct $i,j\in[4]$.
\end{proof}

\begin{lemma}\label{thm11}
A graph $\Sigma_4^*$  contains a vertex $v$, such that starting from $v$, we can find four Hamiltonian paths with different double signs if and only if $\Sigma_4^*$ does not contain exactly three edges that share a common vertex or form a triangle with double sign $g$, for each $g\in \mathbb{F}_2^2$.
\end{lemma}
 \begin{proof} Suppose $\Sigma_4^*$ does not contain exactly three edges that share a common vertex or form a triangle with double sign $g$, for each $g\in \mathbb{F}_2^2$. By Lemma \ref{lemma4}, the multiset  $\sigma(\mathscr{P}_i)$  is equal to $\{x,x,y,y,z,z\}$ or $\{x,y,z,r,r,r\}$ for some distinct $x,y,z,r\in \mathbb{F}_2^2$. If any $\sigma(\mathscr{P}_i)=\{x,y,z,r,r,r\}$ for $i=1,2,3,4$, then, we are done. Otherwise, we may assume none of $\sigma(\mathscr{P}_i)$ are equal to $\{x,y,z,r,r,r\}$. Suppose $\sigma(\mathscr{P}_1)=\{x,x,y,y,z,z\}$. By the assumption and Lemma \ref{same}, $\sigma(\mathscr{P}_1)$, $\sigma(\mathscr{P}_2)$, $\sigma(\mathscr{P}_3)$ and $\sigma(\mathscr{P}_4)$ are distinct. Without loss of generality, we may assume $\sigma(\mathscr{P}_2)=\{x,x,y,y,r,r\}$, $\sigma(\mathscr{P}_3)=\{x,x,r,r,z,z\}$, and $\sigma(\mathscr{P}_4)=\{r,r,y,y,z,z\}$. Put all above double signs together, we can see that there are 24 double signs of the paths. There are six of each $x,y,z,r$.  Each Hamiltonian path in \(\Sigma_4^*\) has been counted twice. Thus, there are exactly three Hamiltonian paths in \(\Sigma_4^*\) with double sign \(x\). This contradicts the Key Lemma. Hence, there must exist a vertex \(v\) such that, starting from \(v\), we can find four Hamiltonian paths with distinct double signs.

Conversely, suppose there exists
 $g \in \mathbb{F}_2^2$  such that  $\Sigma_4^*$  contains exactly three edges that share a common vertex or form a triangle with a double sign  $g$. By Lemma \ref{same}, $\Sigma_4^*$ does not contain a Hamiltonian path with a double sign $g$. That is, for every vertex \(v \in V(\Sigma_4^*)\), there do not exist four Hamiltonian paths starting at \(v\) with distinct double signs.
\end{proof}
 
Finally, we are ready to prove the CASE $\beta$. 

\begin{lemma} (CASE $\beta$)  If $\Sigma_n$ contains a  subgraph $\Sigma_4^*$ and $n>5$. Then,  $\Sigma_n$  contains all possible double signs of Hamiltonian circles.
\end{lemma}

\begin{proof}
We proceed by considering three cases.

\medskip

\textbf{Case 1.} Suppose there exists a vertex, say \(v_5\), not in \(\Sigma_4^*\) such that, after normalizing \(v_5\), the subgraph \(\Sigma_4^*\) becomes \(\Sigma_4^{*'}\), and \(\Sigma_4^{*'}\) does not contain three edges that either share a common vertex or form a triangle with double sign \(x\), for any \(x \in \mathbb{F}_2^2\).
 Then, by Lemma \ref{thm11}, there exists a vertex, say \(v_1\), in \(\Sigma_4^{*'}\) such that the six Hamiltonian paths \(p_1', p_2', \ldots, p_6'\) starting from \(v_1\) in \(\Sigma_4^{*'}\) yield all possible double signs. Connect each ending point of $p_i'$ with $v_5$ to form a path $p_i$ for $i=1,2...,6$. Then, $\sigma(p_i)=\sigma(p_i')$. 
Without loss of generality, we may assume $p_1,p_2,p_3,p_4$ have distinct double signs.  Label the remaining three vertices of $\Sigma_4^{*'}$ as \(v_2\), \(v_3\), and \(v_4\). Denote the induced subgraph on \(\{v_1,v_2,v_3,v_4,v_5\}\) by \(K_5^1\). Relabel the vertices of \(\Sigma_n\) as \(v_1,v_2,v_3,v_4,v_5,\dots,v_n\) and denote the path \(v_5v_6 \cdots v_nv_1\) by \(p\). Connect $K_5^1$ with the path $p$ to form a necklace, see Figure ~\ref{K5}.

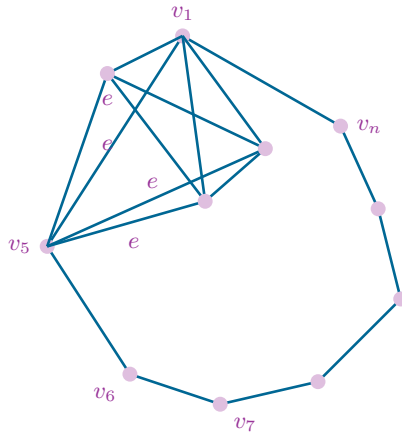
\begin{figure}[H]
    \centering
\begin{tikzpicture}[
  scale=1,
  dot/.style={circle,fill=violet!25,draw=none,inner sep=2pt},
  vtx/.style={circle,fill=violet!25,draw=none,inner sep=2pt},
  ed/.style={line width=1pt, draw=MidnightBlue},
  lbl/.style={font=\scriptsize\itshape, text=violet!80},
  every node/.style={font=\scriptsize}
]

\coordinate (v1) at (-.5,2.6);
\coordinate (vn) at (1.6,1.4);
\coordinate (r2) at (2.1,0.3);
\coordinate (r1) at (2.4,-0.9);
\coordinate (v8) at (1.3,-2.0);
\coordinate (v7) at (0.0,-2.3);
\coordinate (v6) at (-1.2,-1.9);
\coordinate (v5) at (-2.3,-0.2);

\draw[ed]
  (v5)--(v6)--(v7)--(v8)--(r1)--(r2)--(vn)--(v1)--(v5);

\node[vtx] at (v1) {};
\node[vtx] at (vn) {};
\node[vtx] at (r2) {};
\node[vtx] at (r1) {};
\node[vtx] at (v8) {};
\node[vtx] at (v7) {};
\node[vtx] at (v6) {};
\node[vtx] at (v5) {};

\node[lbl, above=2pt]  at (v1) {$v_1$};
\node[lbl, right=2pt]  at (vn) {$v_n$};
\node[lbl, below right=1pt] at (v7) {$v_7$};
\node[lbl, below left=1pt]  at (v6) {$v_6$};
\node[lbl, left=2pt]   at (v5) {$v_5$};

\coordinate (a) at (-1.5,2.1);
\coordinate (b) at (0.6,1.1);
\coordinate (c) at (-0.2,0.4);

\draw[ed] (a)--(b);
\draw[ed] (a)--(c);
\draw[ed] (b)--(c);
\draw[ed] (v1)--(a);
\draw[ed] (v1)--(b);
\draw[ed] (v1)--(c);
\draw[ed] (v5)--(a);
\draw[ed] (v5)--(b);
\draw[ed] (v5)--(c);

\node[vtx] at (a) {};
\node[vtx] at (b) {};
\node[vtx] at (c) {};

\node[lbl] at (-1.5,1.75) {$e$};
\node[lbl] at (-1.5,1.15) {$e$};
\node[lbl] at (-0.9,0.65) {$e$};
\node[lbl] at (-1.15,-0.15) {$e$};

\end{tikzpicture}

 \caption{A single \(K_5\) necklace}
    \label{K5}
\end{figure}

\emph{Claim:} The above necklace contains Hamiltonian circles with all possible double signs in \(\mathbb{F}_2^2\). To see this, for each \(i \in \{1,2,3,4\}\), the union \(p_i \cup p\) forms a Hamiltonian circle in \(\Sigma_n\). The double sign of \(p_i \cup p\) is given by
\[
\sigma(p_i \cup p) = \sigma(p_i) + \sigma(p),
\]
and since \(\sigma(p)\) is fixed (say, equal to \(g\)) and \(g \mathbb{F}_2^2 = \mathbb{F}_2^2\), it follows that
\[
\{\sigma(p_i)+ \sigma(p) : i\in\{1,2,3,4\}\} = \mathbb{F}_2^2.
\]

\medskip

\textbf{Case 2.} Suppose that for each vertex not in \(\Sigma_4^*\), when it is normalized one at a time,  
the subgraph \(\Sigma_4^*\) becomes \(\Sigma_4^{**}\), where \(\Sigma_4^{**}\) contains exactly three edges that either share a common vertex or form a triangle with a common double sign \(x\) for some \(x \in \mathbb{F}_2^2\).
 Pick a vertex not in \(\Sigma_4^*\) and call it \(v_5\), then normalize it. The subgraph \(\Sigma_4^*\) becomes \(\Sigma_4^{**}\) with vertices \(v_1,v_2,v_3,v_4\). Denote the induced subgraph on \(\{v_1,v_2,v_3,v_4,v_5\}\) by \(K_5^1\). Now, assume there is another vertex \(v_6\) such that not all of the edges \(v_1v_6, v_2v_6, v_3v_6,\) and \(v_4v_6\) share the same double sign. Without loss of generality, assume \(\sigma(v_1v_6) \neq \sigma(v_2v_6)\), and denote \(\sigma(v_1v_6)=r\) and \(\sigma(v_2v_6)=t\) for some distinct \(r,t\in \mathbb{F}_2^2\). By Lemma \ref{same}, there exist three distinct double signs \(x,y,z\) among the six Hamiltonian paths starting at \(v_1\) in $\Sigma_4^{*'}$; choose three such paths \(p_1,p_2,p_3\) with \(\sigma(p_1)=x\), \(\sigma(p_2)=y\), and \(\sigma(p_3)=z\). Connect the ending vertices  of these paths to \(v_5\) to form new paths \(p_1',p_2',p_3'\), which are three Hamiltonian paths in $K_5^1$; since the connecting edge from the ending vertex of $p_i'$ to \(v_5\) has double sign \(e\), we have \(\sigma(p_i') = \sigma(p_i)\) for \(i=1,2,3\). Similarly, by Lemma \ref{same} in $K_5^1$, we can construct three Hamiltonian paths \(p_4',p_5',p_6'\) starting at \(v_2\) and ending at \(v_5\) with double signs \(x,y,z\), respectively.

Now, consider the Hamiltonian circles formed as follows. The path \(p:=v_5v_7v_8\cdots v_nv_6\) has a fixed double sign, say \(l\) (i.e., \(\sigma(v_5v_7v_8\cdots v_nv_6)=l\)). Then, the Hamiltonian circle formed by \(v_1v_6\), $p$ and one of the paths \(p_i'\) has a double sign
 $l\,r\,\sigma(p_i)$ for $i=1,2,3$,

and the Hamiltonian circle formed by \(v_2v_6\), $p$ and one of the paths \(p_j'\) has double sign
 $l\,t\,\sigma(p_j')$ for $j=4,5,6$.

Thus, these six Hamiltonian circles above have double signs \(l+r+x\), \(l+r+y\), \(l+r+z\), \(l+t+x\), \(l+t+y\), and \(l+t+z\). By the properties of the group \(\mathbb{F}_2^2\), at least one of \(l+t+x\), \(l+t+y\), or \(l+t+z\) must be distinct from all of \(l+r+x\), \(l+r+y\), and \(l+r+z\), ensuring that every double sign in \(\mathbb{F}_2^2\) appears among all Hamiltonian circles in $\Sigma_n$.

\medskip

\textbf{Case 3.} Use the same setup as in Case~2. Now assume that for each \(i\in\{6,7,\ldots,n\}\) the four edges
\(v_1v_i, v_2v_i, v_3v_i, v_4v_i\) all have the same double sign.

\textbf{Subcase 3a.} Assume $n=6.$ 
Partition the double signs of six edges of \(\Sigma_4^*\) into two sets $A=\{x,x,x\}$ and $B\{y,z,w\}$, where $x,y,z,w\in\mathbb{F}_2^2$ are distinct. Observe that any two edges in \(\Sigma_4^{**}\) with distinct double signs from $B$ form a length-2 path. There are three such paths; denote them as $p_{wy}, p_{yz},p_{wz}$. Pick any two edges in \(\Sigma_4^{**}\) with double sign $x$, then these two edges form a path $p_{xx}$ with double sign $e$. Notice that $p_{wy}, p_{yz},p_{wz}$ and $p_{xx}$ have distinct double signs. 

We now embed each of these paths in a Hamiltonian circle of \(\Sigma_n\).
For example, take \(p_{wy}\) and let \(v_{wy}\) be the unique vertex of
\(V(\Sigma_4^{**})\setminus V(p_{wy})\).
Connect one endpoint of \(p_{wy}\) to \(v_6\) (denote this edge as $e_1$) and the other endpoint of \(p_{wy}\) to \(v_5\) (denote as $e_3$).  
Then connect \(v_6\) to \(v_{wy}\) (denote it $e_{2}$) and \(v_{wy}\) to \(v_5\) (denote it as $e_4$).  By our assumption $\sigma(e_1)=\sigma(e_3)$, and $\sigma(e_2)=\sigma(e_4)$
Thus, the resulting circle is a Hamiltonian circle \(H_{wy}\) in \(\Sigma_n\) and $$\sigma(H_{wy})=\sigma(e_1)+\sigma(e_3)+\sigma(e_2)+\sigma(e_4)+\sigma(p_{wy})=\sigma(p_{wy}).$$
Similarly we obtain Hamiltonian circles \(H_{wz}, H_{yz}, H_{xx}\) with
\(\sigma(H_{wz})=\sigma(p_{wz})\),
\(\sigma(H_{yz})=\sigma(p_{yz})\),
\(\sigma(H_{xx})=\sigma(p_{xx})\).
Since \(p_{wy}, p_{yz}, p_{wz}, p_{xx}\) have distinct double signs, so do
\(H_{wy}, H_{wz}, H_{yz}, H_{xx}\).

\textbf{Subcase 3b.} Assume $n>6.$                                         First of all, we want to construct a path $p_{45}$. If $n=7$, then $p_{45}:=v_4v_5$. If $n>7$, then   $p_{45}:=v_4v_8v_9...v_nv_5$.    Observe that among the six edges in \(\Sigma_4^{**}\), four of them have distinct double signs.
Denote these four edges by \(e_1, e_2, e_3,\) and \(e_4\).  
We aim to construct four Hamiltonian circles \(H_1, H_2, H_3,\) and \(H_4\) such that \(H_i\) contains the edge \(e_i\) for each \(i=1,2,3,4\).  
Without loss of generality, assume that \(e_1:=v_1v_4\) is incident to \(v_4\); then we construct the Hamiltonian circle \(H_1\) in the left panel of Figure~\ref{fig:two-sketches}.
Assume $e_2:=v_1v_2$ does not connect to $v_4$, we construct a Hamiltonian circle $H_2$ as in the right panel of figure \ref{fig:two-sketches}. We construct $H_3,H_4$ similarly.   Notice that all edges from \(v_6\) to \(\{v_1,v_2,v_3,v_4\}\) have the same double sign, and all edges from \(v_7\) to \(\{v_1,v_2,v_3,v_4\}\) have the same double sign.   Thus, $$\sigma(H_1)=\sigma(p_{45})+\sigma(e_1)+\sigma(v_1v_7)+\sigma(v_7v_3)+\sigma(v_3v_6)+\sigma(v_6v_2)+\sigma(v_2v_5)=\sigma(p_{45})+\sigma(e_1).$$
Similarly, \begin{align*}
\sigma(H_2) &= \sigma(p_{45})+\sigma(e_2),\\
\sigma(H_3) &= \sigma(p_{45})+\sigma(e_3),\\
\sigma(H_4) &= \sigma(p_{45})+\sigma(e_4).
\end{align*}

Since the double signs of $e_1,e_2,e_3$ and $e_4$ are distinct, the double signs of $H_1,H_2,H_3 $ and $H_4$ are distinct. 

\begin{figure}[H]
\centering

\begin{minipage}{0.45\textwidth}
\centering
\begin{tikzpicture}[
  scale=1.05,
  vtx/.style={circle,fill=black,inner sep=1.6pt},
  ed/.style={line width=0.9pt},
  dasheded/.style={line width=0.9pt, dashed},
  rededge/.style={line width=2.2pt, draw=red!70},
  lbl/.style={font=\scriptsize}
]
\coordinate (v1) at (0,2);
\coordinate (v4) at (0,0.3);
\coordinate (v2) at (2.0,2);
\coordinate (v6) at (3.2,1.2);
\coordinate (v3) at (2.1,0.3);
\coordinate (v7) at (1.0,-0.8);
\coordinate (v8) at (-0.4,0.6);
\coordinate (v9) at (-1,1.3);
\coordinate (vn) at (-0.3,3);
\coordinate (v5) at (1.1,3);

\draw[ed] (v3)--(v6);
\draw[ed] (v4)--(v8)--(v9);
\draw[ed] (v1)--(v7)--(v3);
\draw[ed] (v2)--(v5);
\draw[dasheded] (v9)--(vn);
\draw[ed] (vn)--(v5);
\draw[ed] (v2)--(v6);

\draw[rededge] (v1)--(v4);

\node[vtx] at (v1) {}; \node[lbl, left=2pt] at (v1) {$v_1$};
\node[vtx] at (v2) {}; \node[lbl, right=1pt] at (v2) {$v_2$};
\node[vtx] at (v3) {}; \node[lbl, right=1pt] at (v3) {$v_3$};
\node[vtx] at (v4) {}; \node[lbl, below=1pt] at (v4) {$v_4$};
\node[vtx] at (v5) {}; \node[lbl, above=1pt] at (v5) {$v_5$};
\node[vtx] at (v6) {}; \node[lbl, right=1pt] at (v6) {$v_6$};
\node[vtx] at (v7) {}; \node[lbl, below=1pt] at (v7) {$v_7$};
\node[vtx] at (v8) {}; \node[lbl, below left=1pt] at (v8) {$v_8$};
\node[vtx] at (v9) {}; \node[lbl, left=1pt] at (v9) {$v_9$};
\node[vtx] at (vn) {}; \node[lbl, left=1pt] at (vn) {$v_n$};
\end{tikzpicture}
\end{minipage}%
\hfill
\begin{minipage}{0.45\textwidth}
\centering
\begin{tikzpicture}[
  scale=1.05,
  vtx/.style={circle,fill=black,inner sep=1.6pt},
  ed/.style={line width=0.9pt},
  dasheded/.style={line width=0.9pt, dashed},
  rededge/.style={line width=2.2pt, draw=red!70},
  lbl/.style={font=\scriptsize}
]
\coordinate (v1) at (0,2);
\coordinate (v4) at (0,0.3);
\coordinate (v2) at (2.0,2);
\coordinate (v6) at (3.2,1.2);
\coordinate (v3) at (2.1,0.3);
\coordinate (v7) at (1.0,-0.8);
\coordinate (v8) at (-0.4,0.6);
\coordinate (v9) at (-1,1.3);
\coordinate (vn) at (-0.3,3);
\coordinate (v5) at (1.1,3);

\draw[ed] (v3)--(v6);
\draw[ed] (v4)--(v8)--(v9);
\draw[ed] (v4)--(v7)--(v3);
\draw[ed] (v1)--(v5);
\draw[dasheded] (v9)--(vn);
\draw[ed] (vn)--(v5);
\draw[ed] (v2)--(v6);


\draw[rededge] (v1)--(v2);

\node[vtx] at (v1) {}; \node[lbl, left=2pt] at (v1) {$v_1$};
\node[vtx] at (v2) {}; \node[lbl, right=1pt] at (v2) {$v_2$};
\node[vtx] at (v3) {}; \node[lbl, right=1pt] at (v3) {$v_3$};
\node[vtx] at (v4) {}; \node[lbl, below=1pt] at (v4) {$v_4$};
\node[vtx] at (v5) {}; \node[lbl, above=1pt] at (v5) {$v_5$};
\node[vtx] at (v6) {}; \node[lbl, right=1pt] at (v6) {$v_6$};
\node[vtx] at (v7) {}; \node[lbl, below=1pt] at (v7) {$v_7$};
\node[vtx] at (v8) {}; \node[lbl, below left=1pt] at (v8) {$v_8$};
\node[vtx] at (v9) {}; \node[lbl, left=1pt] at (v9) {$v_9$};
\node[vtx] at (vn) {}; \node[lbl, left=1pt] at (vn) {$v_n$};
\end{tikzpicture}
\end{minipage}

\caption{ Left: red edge \(e_1:=v_1v_4\). Right: red edge  \(e_2:=v_1v_2\).}
\label{fig:two-sketches}
\end{figure}
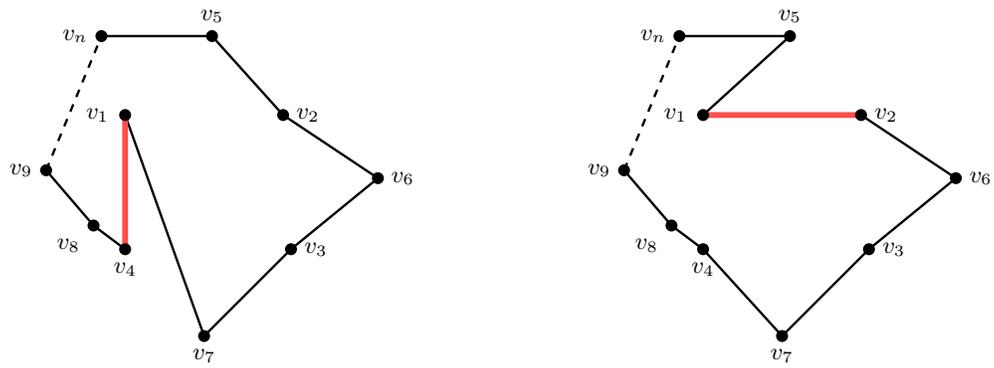
\end{proof}

\newpage

\section{Acknowledgment}

My sincere thanks go to Professor Thomas Zaslavsky for his guidance and many helpful suggestions while preparing this draft.


\end{document}